\documentclass{conm-p-l}

\usepackage{amsthm,amsfonts, cite}
\usepackage{amssymb,graphicx,color}
\usepackage[all]{xy}
\usepackage{mathpazo}
\usepackage{eucal}
\usepackage{enumerate}
\usepackage{hyperref}
\numberwithin{equation}{section}

\usepackage{tikz-cd, float}

\theoremstyle{plain}
\newtheorem{teo}{Theorem}[section]
\newtheorem*{teo*}{Theorem}
\newtheorem{teoA}{Theorem}
\newtheorem{corA}[teoA]{Corollary}

\newtheorem*{cor*}{Corollary}
\newtheorem{lem}[teo]{Lemma}
\newtheorem*{lem*}{Lemma}
\newtheorem{prop}[teo]{Proposition}
\newtheorem{defn}[teo]{Definition}
\newtheorem*{prop*}{Proposition}
\newtheorem{exa}[teo]{Example}

\theoremstyle{remark}
\newtheorem{obs}[teo]{Remark}

\newcommand{\R}{\ensuremath{{\mathbb{R}}}}

\newcommand{\esf}{\ensuremath{{\mathbb{S}}}}
\newcommand{\g}{\ensuremath{\mathtt{g}}}

\renewcommand\vec[1]{\boldsymbol{#1}}

\newcommand\Y{\ensuremath{\mathbf{Y}}}

\makeatletter
\@namedef{subjclassname@2020}{\textup{2020} Mathematics Subject Classification}
\makeatother

\begin{document}

\title[The submanifold compatibility equations in magnetic geometry]{The submanifold compatibility equations in magnetic geometry}

\author[I. Terek]{Ivo Terek}

\address{\parbox{\linewidth}{Department of Mathematics \\ University of California, Riverside \\ Riverside, CA 92521, USA \\ }}

\email{ivo.terek@ucr.edu}


\begin{abstract}
With the notions of magnetic curvature and magnetic second fundamental form recently introduced by Assenza and Albers--Benedetti--Maier, respectively, we establish analogues of the Gauss, Ricci, and Codazzi--Mainardi compatibility equations from submanifold theory in the magnetic setting.
\end{abstract}
\maketitle

\section{Introduction}

A generalization of Riemannian geometry, of much interest in differential geometry and dynamical systems, is \emph{magnetic geometry}: instead of considering just a Riemannian metric $\g$ on a smooth manifold $M$, we consider a pair $(\g,\sigma)$, where $\g$ is a Riemannian metric and $\sigma$ is a closed $2$-form on $M$. Such a pair is called a \emph{magnetic system} on $M$ (in particular, $\sigma$ is called the \emph{magnetic form}), and the dynamics induced by so-called \emph{Landau-Hall equation}, that is, the non-homogeneous differential equation
\begin{equation}\label{eq:landau-hall}
  \frac{{\rm D}\dot{\gamma}}{{\rm d}t}(t) = \Y_{\gamma(t)}(\dot{\gamma}(t))
\end{equation}imposed on smooth curves $\gamma\colon I\to M$, models the motion of particles moving in $M$ subject to the action of the magnetic field corresponding to $\sigma$. Here, ${\rm D}/{\rm d}t$ is the covariant derivative operator along $\gamma$ induced by the Levi-Civita connection $\nabla$ of $\g$, and $\Y\colon TM\to TM$ is the \emph{Lorentz force operator} built from the pair $(\g,\sigma)$, being characterized by the relation
\begin{equation}\label{defn:Lorentz-force}
\g_x(\Y_x(v),w) = \sigma_x(v,w), \quad\mbox{for all $x\in M$ and $v,w\in T_xM$.}  
\end{equation}
 Closedness of $\sigma$ is, of course, nothing more than Gauss's law, stating that magnetic fields are divergence-free (or in more physical terms: magnetic monopoles are not observed in nature). This formalism goes back to Anosov and Sinai \cite{AnosovSinai_1967} and Arnold \cite{Arnold_1961} in the 1960s, but it is still a very active topic of research---see, e.g., \cite{AMP_2015, AMMP_2017, AB_2021, BK_2022}.

For a very long time, a missing part of the puzzle was the notion of curvature for magnetic systems: as soon as $\sigma$ does not identically vanishes, solutions of \eqref{eq:landau-hall} are not the actual geodesics for any linear connection on $M$ \cite[Proposition 2.1]{GLH_2005}. Therefore, there is no curvature tensor coming from a connection which may properly control geometric (or dynamic) properties of the magnetic system $(\g,\sigma)$. In the surface case, the definition of magnetic Gaussian curvature is due to G. and M. Paternain \cite{PP96}. In the higher-dimensional case, the definitions of magnetic curvature operator, sectional curvature, and Ricci curvature, are due to Assenza \cite{Assenza_2024}; earlier work by Gouda \cite{Gouda-97, Gouda-98} (focusing on uniform magnetic systems, that is, those with $\nabla\sigma=0$) and Wojtkowski \cite{Wojtkowski_2000} also hinted at them. We review these definitions in detail in Section \ref{sec:magnetic_curvature}.

With suitable definitions of magnetic curvature in place, the way is open for many more classical results in Riemannian geometry to be generalized to the magnetic setting. For instance, magnetic versions of Synge's theorem \cite[Theorem C]{Assenza_2024} and of the Bonnet-Myers theorem \cite[Lemma 14]{Assenza_2024} have already been established. Green's theorem \cite{Green_1958} has also been generalized \cite[Theorem A]{ART_2024}, and nontrivial magnetic systems with constant magnetic sectional curvature have been characterized: they are all magnetically flat K\"{a}hler magnetic systems, with the underlying metric having constant holomorphic sectional curvature, which is necessarily negative outside of the surface case \cite[Theorem D]{ART_2024}; such K\"{a}hler systems have been extensively investigated by Adachi \cite{Adachi_2014, Adachi_2012,Adachi_1997, Adachi_1995, Adachi_2019,AdachiBai_2013}.

In the recent work \cite{ABM_2025} of Albers--Benedetti--Maier, studying the magnetic dynamics of round odd-dimensional spheres $\esf^{2n+1}$ with the magnetic form induced by the standard contact structure of $\esf^{2n+1}$, they are led to consider the \emph{totally magnetic submanifolds} of $\esf^{2n+1}$: the submanifolds $N\subseteq \esf^{2n+1}$ with the property that magnetic geodesics---that is, solutions of \eqref{eq:landau-hall}---which start tangent to $N$ remain in $N$ for sufficiently small time (that is, the magnetic analogues of totally geodesic submanifolds in Riemannian geometry). In \cite[Theorem 1.4]{ABM_2025} they present a general criterion for a submanifold to be totally magnetic, in terms of its second fundamental form and Lorentz force, and they completely describe the totally magnetic submanifolds of $\esf^{2n+1}$ up to \emph{magnetomorphisms}, which are the diffeomorphisms preserving both the metric and magnetic form.

The criterion just mentioned was originally presented in an earlier preprint \cite{ABM_2025-v1} with the aid of a suitable \emph{magnetic second fundamental form}, which brings us to the goal of the present paper: using the definitions of magnetic curvature for higher-dimensional systems, introducing a definition of magnetic shape operator, and extending the definition of magnetic second fundamental form given in \cite{ABM_2025-v1}, we establish magnetic analogues of the classical compatibility equations of submanifold theory (Gauss, Ricci, and Codazzi--Mainardi).

\begin{teoA}\label{mag-comp}
  Let $(\g,\sigma)$ be a magnetic system on a smooth manifold $M$, and \mbox{$N\subseteq M$} be a submanifold of $M$, equipped with its induced magnetic system. The $s$-magnetic curvature operator of $M$ is decomposed as
\[
  \begin{split}
    {\rm (i)}~ (\vec{M}^{\g,\sigma}_s)^M_{(x,v)}(w)^\top &= (\vec{M}^{\g,\sigma}_s)^N_{(x,v)}(w)  -S_{({\rm II}^\sigma_s)_{(x,v)}(v)}(w)+ (S^\sigma_s)_{(x,v)}({\rm II}_x(v,w))\\ &\quad  - \frac{1}{2} {\rm P}_{v^\perp}\left(sS_{\Y^\perp_x(w)}(v) + s\Y^M_x({\rm II}_x(v,w))^\top + \frac{1}{2}\Y^M_x(\Y^\perp_x(w))^\top\right) \\ {\rm (ii)}~  (\vec{M}^{\g,\sigma}_s)^M_{(x,v)}(w)^\perp &= (\nabla^\perp_w{\rm II}^\sigma_s)_{(x,v)}(v) - (\nabla^\perp_v{\rm II}^\sigma_s)_{(x,v)}(w) - \frac{s}{2} (\nabla^\perp_v\Y^\perp)_x(w) \\ &\quad -s{\rm II}_x(w, \Y^N_x(v)) + \frac{s}{2} \Y^M_x({\rm II}_x(v,w))^\perp + \frac{s}{2} {\rm II}_x(v,\Y^N_x(w)) \\ &\quad -\frac{3}{4} \Y^\perp_x({\rm P}_v(\Y^N_x(w))) - \frac{1}{4}\Y^\perp_x(\Y^N_x(w)) - \frac{1}{4}\Y^M_x(\Y_x^\perp(w))^\perp
  \end{split}
\]for all $(x,v) \in SN$, $w\in T_xN \cap v^\perp$, and $s>0$.  
\end{teoA}

The technical definitions of the objects ${\rm II}^\sigma_s$, $S^\sigma_s$, and $\Y^\perp$ appearing in Theorem \ref{mag-comp} (magnetic second fundamental form, magnetic shape operator, normal Lorentz force) are given in Section \ref{mag-II}. This immediately gives us a relation between the magnetic sectional curvatures of $M$ and $N$:

\begin{corA}\label{mag-sec-Gauss}
  With the setup of Theorem \ref{mag-comp}, we have that
  \[
    \begin{split}
({\rm sec}^{\g,\sigma}_s)_x^M(v,w) &= ({\rm sec}^{\g,\sigma}_s)_x^N(v,w) - \g_x\big(({\rm II}^\sigma_s)_{(x,v)}(v), {\rm II}_x(w,w)\big) \\ &\qquad + \g_x\big({\rm II}_x(v,w), ({\rm II}^\sigma_s)_{(x,v)}(w)\big) + \frac{1}{4} \|\Y^\perp_x(w)\|^2      
    \end{split}
  \]for every $s>0$ and $(x,(v,w)) \in {\rm St}_2(N,\g)$.
\end{corA}

In particular, the codimension-one case deserves some attention:

\begin{corA}\label{cor-hyp}
Let $(\g,\sigma)$ be a magnetic system on a smooth manifold $M$, and \mbox{$N\subseteq M$} be a two-sided hypersurface equipped with its induced magnetic system and a global unit normal field $\eta\in \mathfrak{X}^\perp(N)$. Then we have that \[
  \begin{split}
{\rm (i)} ~ ({\rm sec}^{\g,\sigma}_s)^N_x(v,w) &= ({\rm sec}^{\g,\sigma}_s)^M_x(v,w) \\ &\quad  +
    \begin{vmatrix}
       \g_x((S^\sigma_s)_{(x,v)}(\eta_x),v)  & \g_x(S_\eta(v), w) \\ \g_x((S^\sigma_s)_{(x,v)}(\eta_x),w) & \g_x(S_\eta(w),w)
    \end{vmatrix} - \frac{1}{4} \theta_x(w)^2,    \\[1em] {\rm (ii)}~ ({\rm Ric}^{\g,\sigma}_s)^N(x,v)&= ({\rm Ric}^{\g,\sigma}_s)^M(x,v) - ({\rm sec}^{\g,\sigma}_s)_x^M(v,\eta_x) \\ &\hspace{-1.4em}+ {\rm tr}(S_\eta) \g_x((S^\sigma_s)_{(x,v)}(\eta_x), v)  -\g_x((S^\sigma_s)_{(x,v)}(\eta_x), S_\eta(v)) - \frac{1}{4} \|\theta|_{v^\perp}\|^2,
  \end{split}
\]where $\theta\in \Omega^1(N)$ is defined by $\theta_x(w) = \g_x(\Y^\perp_x(w),\eta_x)$.
\end{corA}

One may also consider and compute the remaining six projected quantities \begin{equation}\label{remaining-projected}
  \begin{split}
    &(\vec{M}^{\g,\sigma}_s)^M_{(x,v)}(\eta)^\top,\quad (\vec{M}^{\g,\sigma}_s)^M_{(x,v)}(\eta)^\perp, \quad (\vec{M}^{\g,\sigma}_s)^M_{(x,\xi)}(w)^\top, \\ &(\vec{M}^{\g,\sigma}_s)^M_{(x,\xi)}(w)^\perp, \quad (\vec{M}^{\g,\sigma}_s)^M_{(x,\xi)}(\eta)^\top,\quad\mbox{and}\quad(\vec{M}^{\g,\sigma}_s)^M_{(x,\xi)}(\eta)^\perp,
  \end{split}\end{equation}
where $v,w$ are tangent to $N$, and $\xi,\eta$ are normal to $N$. The main reason we do not pursue it further here is technical: the definition of magnetic curvature operator is made as an extension of the one for \emph{tidal force operators} (in the sense of \cite[Definition 8, p. 219]{ONeill}), built from the Riemann curvature tensor and appearing prominently in the Jacobi equation: $F_{(x,v)}\colon v^\perp \to v^\perp$, given by $F_{(x,v)}(w) = R_x(w,v)v$. The quantities \eqref{remaining-projected} would then involve the terms $R^M_x(\eta,v)v$, $R^M_x(w,\xi)\xi$, and $R^M_x(\eta,\xi)\xi$, none of which can be directly projected using the classical compatibility equations \eqref{Gauss-eqn}--\eqref{Ricci-eqn}.

In Appendix \ref{app-killing}, we apply our results to compute the magnetic sectional and Ricci curvatures of the so-called \emph{Killing magnetic systems} on three-di\-men\-si\-o\-nal manifolds, that is, those whose Lorentz force operator is given by taking cross products with a Killing vector field. The flows of such systems have been computed in several ambient manifolds and received considerable attention in recent research \cite{Cabrerizo_2013, Nistor_2015, DIM_2023, KDA_2023, BEK_2024}, even in the case where the underlying metric is Lorentzian instead of Riemannian \cite{Iqbal_2022, ZLC_2024, EI_2025}. We conclude the discussion by deducing from the behavior of the magnetic curvatures that the Ma\~{n}\'{e} critical value of the Killing magnetic systems on $\esf^3$ whose magnetic vector fields correspond to left-multiplication by some unit quaternion equals $1/2$ (Proposition \ref{Mane-S3}).

\medskip

\noindent \textbf{Acknowledgements.} I would like to thank James Marshall Reber for a very helpful discussion on Ma\~{n}\'{e} critical values and for some feedback on a previous draft of this text. I also appreciate the comments made by the referees, which helped improve the exposition.

\section{Magnetic curvature}\label{sec:magnetic_curvature}
 
We briefly review the several definitions of magnetic curvature presented by Assenza in \cite{Assenza_2024}. Let $M$ be a smooth manifold, $(\g,\sigma)$ be a magnetic system on $M$, and $\Y$ be its Lorentz force operator as in \eqref{defn:Lorentz-force}. We consider the unit sphere bundle $SM\to M$ of $(M,\g)$,
\begin{equation}\label{bundle-E}
  \parbox{.67\textwidth}{the vector bundle $E\to SM$ of orthogonal hyperplanes, whose fiber over an element $(x,v)\in SM$ is $E_{(x,v)} = v^\perp$}
\end{equation}
 (i.e., the $\g$-orthogonal complement of the line spanned by $v$ in $T_xM$), and the Stiefel bundle ${\rm St}_2(M,\g) \to M$ of ordered $\g$-orthonormal $2$-frames tangent to $M$, whose fiber over a point $x\in M$ is the set of all pairs $(v,w)$, with $v,w\in T_xM$ both unit and $\g$-orthogonal.

The endomorphisms $A^{\g,\sigma},R^{\g,\sigma}_s\colon E\to E$, where $s>0$, are defined by
\begin{equation}\label{eqn:AOmega}
	\begin{split}
{\rm (i)}~	A^{\g,\sigma}_{(x,v)}(w) &= -\frac{3}{4} \Y_x({\rm P}_v(\Y_x(w))) - \frac{1}{4} {\rm P}_{v^\perp}(\Y_x^2(w)),\quad\mbox{and}\\ {\rm (ii)}~(R^{\g,\sigma}_s)_{(x,v)}(w) &= s^2 R_x(w,v)v - s(\nabla_w\Y)_x(v)  + \frac{s}{2} {\rm P}_{v^\perp}((\nabla_{v}\Y)_x(w)),   
	\end{split}
\end{equation}
and the \emph{$s$-magnetic curvature operator}, $\vec{M}^{\g,\sigma}_s\colon E\to E$, is then given by
\begin{equation}\label{def:MOmega}
  (\vec{M}^{\g,\sigma}_s)_{(x,v)}(w) = (R^{\g,\sigma}_s)_{(x,v)}(w) + A^{\g,\sigma}_{(x,v)}(w).
\end{equation}
In \eqref{eqn:AOmega}, $\nabla$ denotes the Levi-Civita connection of $\g$, while $R$ is its curvature tensor, and the projections ${\rm P}_v\colon T_xM \to \R v$ and ${\rm P}_{v^\perp}\colon T_xM \to v^\perp$ refer to the \mbox{$\g$-orthogonal} direct-sum decomposition $T_xM = (\R v)\oplus v^\perp$.

With this, the \emph{$s$-magnetic sectional curvature} ${\rm sec}^{\g,\sigma}_s\colon {\rm St}_2(M,\g)\to \R$ and the \emph{\mbox{$s$-magnetic} Ricci curvature} ${\rm Ric}^{\g,\sigma}_s\colon SM\to \R$ are defined by, respectively:
\begin{equation}\label{def:sec_and_Ric_mag}
 ({\rm sec}^{\g,\sigma}_s)_x(v,w) = \g_x\left((\vec{M}^{\g,\sigma}_s)_{(x,v)}(w),w\right)  \hspace{.7em}\mbox{and}\hspace{.7em} {\rm Ric}^{\g,\sigma}_s(x,v) = {\rm tr}\,(\vec{M}^{\g,\sigma}_s)_{(x,v)}.
\end{equation}

The dependence of the objects $R^{\g,\sigma}_s,\vec{M}^{\g,\sigma}_s, {\rm sec}^{\g,\sigma}_s, {\rm Ric}_s^{\g,\sigma}$ on the energy parameter \mbox{$s>0$} reflects non-homogeneity of \eqref{eq:landau-hall}, which in turn implies that the dynamics of the magnetic flow it induces on each radius-$s$ sphere bundle depends heavily on the value of $s$. Finally, $A^{\g,\sigma}$ is simply the part of the magnetic curvature operator which is both quadratic in $\Y$ and insensitive to $s$. We also note that the magnetic sectional curvature does not survive as a function on the Grassmannian bundle \mbox{${\rm Gr}_2(TM) \to M$}, as $({\rm sec}^{\g,\sigma}_s)_x(v,w)$ does not equal $({\rm sec}^{\g,\sigma}_s)_x(w,v)$ in general.

\section{The classical compatibility equations from submanifold theory}

The content of this section is standard and it is only included for the reader's convenience---more details can be found in, for instance, \cite[Section 1.3]{dajczer-tojeiro}. Let $(M,\g)$ be a Riemannian manifold, and $N\subseteq M$ be a submanifold of $M$. We denote by $\nabla^M$ and $\nabla^N$ the Levi-Civita connections of $\g$ and of the metric on $N$ induced by $\g$. For each $x\in N$ we have
\begin{equation}\label{tan-nor-decomp-TN}
  \parbox{.82\textwidth}{the $\g$-orthogonal direct-sum decomposition $T_xM = T_xN \oplus [T_xN]^\perp$,}
\end{equation}
 leading to the Gauss and Weingarten formulas
\begin{equation}\label{Gauss-Weingarten}
{\rm (i)}~  \nabla^M_XY = \nabla^N_XY + {\rm II}(X,Y)\quad\mbox{and}\quad {\rm (ii)}~\nabla^M_X\xi = -S_\xi(X) + \nabla^\perp_X\xi,
\end{equation}for all $X,Y\in \mathfrak{X}(N)$ and $\xi \in \mathfrak{X}^\perp(N)$. Here, ${\rm II}$ and $\nabla^\perp$ denote the second fundamental form and normal connection of $N$ relative to $M$, respectively, while $S_\xi\colon TN\to TN$ denotes the shape operator associated with $\xi$. The second fundamental form and shape operators are related via
\begin{equation}\label{II-S}
 \g({\rm II}(X,Y),\xi) = \g(S_\xi(X),Y).
\end{equation}

Decomposing the curvature tensor of $(M,\g)$, we obtain the Gauss equation
\begin{equation}\label{Gauss-eqn}
  [R^M(X,Y)Z]^\top = R^N(X,Y)Z - S_{{\rm II}(Y,Z)}(X) + S_{{\rm II}(X,Z)}(Y),
\end{equation}the Codazzi-Mainardi equation
\begin{equation}\label{CM-eqn}
  [R^M(X,Y)Z]^\perp =  (\nabla^\perp_X{\rm II})(Y,Z) - (\nabla^\perp_Y{\rm II})(X,Z),
\end{equation}and the Ricci equation
\begin{equation}\label{Ricci-eqn}
  [R^M(X,Y)\xi]^\perp = R^\perp(X,Y)\xi + {\rm II}(S_\xi(X),Y) - {\rm II}(X,S_\xi(Y)),
\end{equation}for all $X,Y,Z\in \mathfrak{X}(N)$ and $\xi \in \mathfrak{X}^\perp(N)$. Here, $R^\perp$ is the curvature tensor of the normal connection $\nabla^\perp$, while the covariant derivative of ${\rm II}$ is defined by the expression $(\nabla^\perp_X{\rm II})(Y,Z) = \nabla_X^\perp({\rm II}(Y,Z)) - {\rm II}(\nabla_X^NY,Z) - {\rm II}(Y,\nabla_X^NZ)$ (that is, it is computed with the van der Waerden-Bortolotti connection $\nabla^N\oplus \nabla^\perp$).

\section{The magnetic second fundamental form and shape operator}\label{mag-II}

Let $(\g,\sigma)$ be a magnetic system on a smooth manifold $M$, and $N\subseteq M$ be a submanifold of $M$. We may restrict $(\g,\sigma)$ to a magnetic system on $N$, so that
\begin{equation}\label{YM-YN}
  \parbox{.61\textwidth}{the Lorentz force operators $\Y^M$ and $\Y^N$ are related via $\Y^N_x(v) = [\Y^M_x(v)]^\top$ for all $x\in N$ and $v\in T_xN$;}
\end{equation}
the orthogonal projections $[\,\cdot\,]^\top$ and $[\,\cdot\,]^\perp$ being taken relative to \eqref{tan-nor-decomp-TN}. For economy of notation, we set $\Y^\perp_x(v) = [\Y^M_x(v)]^\perp$ for $v\in T_xN$. Let $\pi\colon SN\to N$ be the bundle projection. With this setup:

\begin{defn}\label{defn:mag-II}
  Let $s>0$ be any positive real number.
  \begin{enumerate}[\normalfont(a)]
  \item The \emph{$s$-magnetic second fundamental form} of $N$ relative to $M$ is the bundle morphism ${\rm II}^\sigma_s\colon \pi^*(TN) \to [TN]^\perp$ given by $({\rm II}^\sigma_s)_{(x,v)}(w) = s^2{\rm II}_x(v,w) - s \Y^\perp_x(w)$.
  \item The \emph{$s$-magnetic shape operator} of $N$ relative to $M$ is the bundle morphism \linebreak[4]$S^\sigma_s\colon \pi^* ([TN]^\perp) \to TN$ given by $(S^\sigma_s)_{(x,v)}(\xi) = s^2 S_\xi(v) + s \Y^M_x(\xi)^\top$.
  \end{enumerate}
\end{defn}

\begin{obs}
  Our Definition \ref{defn:mag-II}(a) does not immediately agree with \cite[Definition 6.9]{ABM_2025-v1}, which is ${\rm II}^{\rm mag}\colon TN\to [TN]^\perp$ with ${\rm II}^{\rm mag}_x(v) = {\rm II}_x(v,v) + \Y^M_x(v) - \Y^N_x(v)$ for any $v\in T_xN$, by an overall sign and the presence of $s$. There are three main reasons for this, which we justify as follows:
  \begin{enumerate}[(i)]
  \item The sign convention adopted in \cite{ABM_2025-v1} for the classical second fundamental form ${\rm II}$ is the opposite of the standard one used in submanifold theory \cite{ONeill, dajczer-tojeiro, manfredo, lee}, which we adopt here, flipping the sign in the right side of (\ref{Gauss-Weingarten}-i). However, as pointed out in \cite{ABM_2025-v1}, the quantity ${\rm II}_x(v,v) + \Y^M_x(v) - \Y^N_x(v)$ vanishes for every $v$ if and only if ${\rm II}_x(v,v)$ and $\Y^M_x(v) - \Y^N_x(v)$ separately vanish for every $v$, as the former is quadratic in $v$ while the latter is linear in $v$. For the same reason, this also holds for our magnetic second fundamental form, meaning that our $({\rm II}^\sigma_s)_{(x,v)}(v)$ vanishes for all $(x,v)\in SN$ if and only if ${\rm II}^{\rm mag}_x(v)$ in \cite{ABM_2025-v1} does whenever $\|v\|=s$.
  \item The definition of magnetic curvature operator \eqref{eqn:AOmega}--\eqref{def:MOmega} makes the dependence on the parameter $s>0$ explicit, by restricting $(x,v)$ to live in the appropriate unit sphere bundle. As our goal is to relate the magnetic second fundamental form to the magnetic curvature operators of $M$ and $N$, it is natural to carry out this same normalization here. 
  \end{enumerate}
  Points (i) and (ii), together with \cite[Remark 6.10]{ABM_2025-v1}, imply that $({\rm II}^\sigma_s)_{(x,v)}(v) = 0$ for every $(x,v)\in SN$ if and only if $N$ is \emph{totally $s$-magnetic}, in the sense that every magnetic geodesic (that is, a solution of \eqref{eq:landau-hall}) with the specific speed $s$ which starts tangent to $N$ remains in $N$ for small time.
  \begin{enumerate}[(i)]\setcounter{enumi}{2}
  \item The classical second fundmental form may be regarded as a collection of linear transformations ${\rm II}_x(v,\cdot) \colon T_xN \to [T_xN]^\perp$, one for each $(x,v)\in SN$, making it a bundle morphism ${\rm II}\colon \pi^*(TN)\to [TN]^\perp$, and our definition of ${\rm II}^\sigma_s$ only adds a magnetic term to that. Similarly, as the quantity $S_\xi(v)$ is linear in both $v\in T_xN$ \emph{and} $\xi \in [T_xN]^\perp$, it may be regarded as a linear transformation $S(v)\colon [T_xN]^\perp \to T_xN$ (the $\g$-adjoint of ${\rm II}_x(v,\cdot)$), and the definition of $S^\sigma_s$ is made so that the relation \[\g_x\big(({\rm II}^\sigma_s)_{(x,v)}(w), \xi\big) = \g_x\big((S_s^\sigma)_{(x,v)}(\xi), w\big)\]always holds, mimicking \eqref{II-S}. This is also in line with the general philosophy that, once a magnetic field is introduced, geometric objects become \emph{anisotropic} (that is, di\-rec\-ti\-on-dependent), already illustrated by the anisotropic Lorentz force and anisotropic twisted connection used in \cite{Assenza_2024} and \cite{ART_2024}.
  \end{enumerate}
\end{obs}

\begin{exa}\label{tot-s-mag}
  Let $(\g,\sigma)$ be any magnetic system on a smooth manifold $M$. If \mbox{$N\subseteq M$} is any submanifold for which $\Y^\perp=0$, then $N$ is totally $s$-magnetic for some (and hence every) $s>0$ if and only if $N$ is totally geodesic. This is the case, in particular, when $(\g,\sigma) = (\g,\lambda\omega)$ is a K\"{a}hler magnetic system and $N$ is a complex submanifold of $M$, as then $\Y^M = \lambda J$ and $TN$ being $J$-invariant imply that $\Y^\perp=0$.
\end{exa}

\section{Orthogonal decompositions and the proof of Theorem \ref{mag-comp}}

We continue with the setup of the previous section.

Via \eqref{bundle-E}, we now obtain two vector bundles $E^M\to SM$ and $E^N \to SN$; here $SN$ denotes the unit sphere bundle of $N$. The orthogonal projection $TM\to TN$ clearly maps $E^M$ onto $E^N$, cf. Figure \ref{fig:EM-EN}. As a crucial intermediate step in establishing Theorem \ref{mag-comp}, we will first determine how the operators $(A^{\g,\sigma})^M\colon E^M\to E^M$ and $(A^{\g,\sigma})^N\colon E^N\to E^N$---cf. (\ref{eqn:AOmega}-i)---are related.

\begin{figure}[H]
  \centering
  \tikzset{every picture/.style={line width=0.75pt}} %set default line width to 0.75pt        

\begin{tikzpicture}[x=0.75pt,y=0.75pt,yscale=-1,xscale=1]

%Straight Lines [id:da4886256292048132] 
\draw    (210.5,162) -- (210,72.5) ;
%Straight Lines [id:da0000862038788052022] 
\draw  [dash pattern={on 4.5pt off 4.5pt}]  (210.5,162) -- (210.5,192.5) ;
%Straight Lines [id:da2924662652834775] 
\draw    (210.5,192.5) -- (210,270) ;
%Straight Lines [id:da09082207342520054] 
\draw    (210.5,168.5) -- (160.94,178.41) ;
\draw [shift={(158,179)}, rotate = 348.69] [fill={rgb, 255:red, 0; green, 0; blue, 0 }  ][line width=0.08]  [draw opacity=0] (10.72,-5.15) -- (0,0) -- (10.72,5.15) -- (7.12,0) -- cycle    ;
%Straight Lines [id:da8280239793038657] 
\draw [color={rgb, 255:red, 38; green, 0; blue, 251 }  ,draw opacity=1 ][line width=1.5]    (180.25,151.25) -- (254.75,192) ;
%Straight Lines [id:da41944588916232584] 
\draw [color={rgb, 255:red, 255; green, 0; blue, 0 }  ,draw opacity=1 ]   (180.25,71.25) -- (254.75,112) ;
%Straight Lines [id:da259921477132232] 
\draw [color={rgb, 255:red, 255; green, 0; blue, 0 }  ,draw opacity=1 ]   (180.25,71.25) -- (180.25,151.25) ;
%Straight Lines [id:da9434470068406782] 
\draw [color={rgb, 255:red, 255; green, 0; blue, 0 }  ,draw opacity=1 ] [dash pattern={on 4.5pt off 4.5pt}]  (180.25,151.25) -- (180.25,192.75) ;
%Straight Lines [id:da9923669141937561] 
\draw [color={rgb, 255:red, 255; green, 0; blue, 0 }  ,draw opacity=1 ]   (254.75,112) -- (254.75,192) ;
%Straight Lines [id:da055952623402383495] 
\draw    (156.3,151) -- (76.5,192.5) ;
%Straight Lines [id:da8154602895210024] 
\draw    (76.5,192.5) -- (262.7,192.5) ;
%Straight Lines [id:da30643253225348377] 
\draw    (262.7,192.5) -- (342.5,151) ;
%Straight Lines [id:da1298945985853035] 
\draw  [dash pattern={on 4.5pt off 4.5pt}]  (180.25,151.25) -- (254.75,152) ;
%Straight Lines [id:da9963753269594333] 
\draw    (156.3,151) -- (180.25,151.25) ;
%Straight Lines [id:da8131759154056692] 
\draw    (254.75,152) -- (342.5,151) ;
%Straight Lines [id:da8901234204148025] 
\draw [color={rgb, 255:red, 255; green, 0; blue, 0 }  ,draw opacity=1 ]   (254.75,192) -- (254.75,233.5) ;
%Straight Lines [id:da5414944650633211] 
\draw [color={rgb, 255:red, 255; green, 0; blue, 0 }  ,draw opacity=1 ]   (254.75,233.5) -- (254.75,275) ;
%Straight Lines [id:da6158489448171489] 
\draw [color={rgb, 255:red, 255; green, 0; blue, 0 }  ,draw opacity=1 ]   (180.25,192.75) -- (180.25,234.25) ;
%Straight Lines [id:da5969463925120753] 
\draw [color={rgb, 255:red, 255; green, 0; blue, 0 }  ,draw opacity=1 ]   (180.25,234.25) -- (254.75,275) ;

% Text Node
\draw (311.5,133.4) node [anchor=north west][inner sep=0.75pt]  [font=\small]  {$T_{x} N$};
% Text Node
\draw (192.5,51.9) node [anchor=north west][inner sep=0.75pt]  [font=\small]  {$[ T_{x} N]^{\perp }$};
% Text Node
\draw (145.5,172.9) node [anchor=north west][inner sep=0.75pt]  [font=\small]  {$v$};
% Text Node
\draw (260,258) node [anchor=north west][inner sep=0.75pt]  [font=\small,color={rgb, 255:red, 255; green, 0; blue, 0 }  ,opacity=1 ]  {$E_{( x,v)}^{M}$};
% Text Node
\draw (145,125) node [anchor=north west][inner sep=0.75pt]  [font=\small,color={rgb, 255:red, 38; green, 0; blue, 251 }  ,opacity=1 ]  {$E_{( x,v)}^{N}$};

% Text Node
\draw (215,158) node [anchor=north west][inner sep=0.75pt]  [font=\small]  {$x$};
\end{tikzpicture}
\caption{The relation $E^N_{(x,v)}= E^M_{(x,v)}\cap T_xN$, for $(x,v)\in SN$.}
  \label{fig:EM-EN} 
\end{figure}

We have the obvious relations
\begin{equation}\label{all-projectors}
  \begin{split}
  {\rm (i)}&~ {\rm P}_{v^\perp}(z^\perp) = {\rm P}_{v^\perp}(z^\perp) = z^\perp,\hspace{.7em} {\rm (ii)}~{\rm P}_{v^\perp}(z^\top) = {\rm P}_{v^\perp}(z)^\top , \\  {\rm (iii)}&~ ({\rm P}_v\circ \Y^M_x)|_{T_xN} = {\rm P}_v\circ \Y^N_x,\quad\mbox{and}\quad {\rm (iv)}~{\rm P}_v\circ \Y^\perp_x=0,
  \end{split}
\end{equation}for any $(x,v)\in SN$ and $z\in T_xM$, where ${\rm P}_v\colon T_xM \to \R v$ and ${\rm P}_{v^\perp}\colon T_xM \to E_{(x,v)}^M$ are as in Section \ref{sec:magnetic_curvature}. Note that (\ref{all-projectors}-iv) follows from (\ref{all-projectors}-iii).

\begin{prop}\label{decomp_Ags}
The formulas
\[ \begin{split}
    {\rm (i)}~ (A^{\g,\sigma})_{(x,v)}^M(w)^\top &= (A^{\g,\sigma})^N_{(x,v)}(w) - \frac{1}{4} {\rm P}_{v^\perp}(\Y^M_x(\Y^\perp_x(w)))^\top \\ {\rm (ii)}~ (A^{\g,\sigma})_{(x,v)}^M(w)^\perp &=-\frac{3}{4} \Y_x^\perp({\rm P}_v(\Y^N_x(w))) - \frac{1}{4} \Y^\perp_x(\Y^N_x(w)) - \frac{1}{4} \Y^M_x(\Y^\perp_x(w))^\perp
  \end{split}\]hold, for any $(x,v) \in SN$ and $w \in E^N_{(x,v)}$.
\end{prop}

\begin{proof}
  As ${\rm P}_v(\Y^M_x(w)) \in T_xN$, we may apply \eqref{YM-YN} and (\ref{all-projectors}-ii) when taking the tangent projection of formula (\ref{eqn:AOmega}-i) for $(A^{\g,\sigma})^M$ to obtain \[  (A^{\g,\sigma})_{(x,v)}^M(w)^\top = -\frac{3}{4} \Y^N_x({\rm P}_v(\Y^N_x(w))) - \frac{1}{4} {\rm P}_{v^\perp}((\Y^M_x)^2(w))^\top.  \]Substituting formula (\ref{eqn:AOmega}-i) for $(A^{\g,\sigma})^N$ into the above and simplifying, we obtain ${\rm (i)}$. As for ${\rm (ii)}$, we proceed similarly: taking the normal projection of formula \mbox{(\ref{eqn:AOmega}-i)} for $(A^{\g,\sigma})^M$, using the definition of $\Y^\perp$ and (\ref{all-projectors}-iv) on the first term in the right side, and (\ref{all-projectors}-i) on the second term, we have that
  \[  (A^{\g,\sigma})^M_{(x,v)}(w)^\perp =  -\frac{3}{4} \Y_x^\perp({\rm P}_v(\Y^N_x(w)) -\frac{1}{4} (\Y^M_x)^2(w)^\perp.  \]Further expanding $(\Y^M_x)^2(w)^\perp = \Y^\perp_x(\Y^N_x(w)) + \Y^M_x(\Y^\perp_x(w))^\perp$, ${\rm (ii)}$ follows.
\end{proof}

Next, we must relate $(R^{\g,\sigma}_s)^M\colon E^M\to E^M$ and $(R^{\g,\sigma}_s)^N\colon E^N\to E^N$. As \mbox{(\ref{eqn:AOmega}-ii)} involves covariant derivatives of the Lorentz force operator, it is convenient to start with an auxilliary result.

\begin{lem}\label{cov-der-YM}
  The formulas
\[  \begin{split}
     {\rm (i)}~ (\nabla^M_v\Y^M)_x(w)^\top  &= (\nabla^N_v\Y^N)_x(w) - S_{\Y^\perp_x(w)}(v) - \Y^M_x({\rm II}_x(v,w))^\top \\ {\rm (ii)}~(\nabla^M_v\Y^M)_x(w)^\perp &= (\nabla^\perp_v\Y^\perp)_x(w) + {\rm II}_x(v,\Y^N_x(w)) - \Y^M_x({\rm II}_x(v,w))^\perp
    \end{split}\]hold, for any $x\in N$ and $v,w\in T_xN$, where the normal covariant derivative $\nabla^\perp\Y^\perp$ is defined as $(\nabla^\perp_X \Y^\perp)(Y) = \nabla_X^\perp(\Y^\perp(Y)) - \Y^\perp(\nabla_X^NY)$, for all $X,Y\in \mathfrak{X}(N)$.
\end{lem}

\begin{proof}
  Formulas (i) and (ii) are tensorial, but to establish them we first choose vector fields $V$ and $W$ tangent to $N$, defined on some neighborhood of $x$, which extend $v$ and $w$. Then we project both sides of the relation
  \[ (\nabla^M_V\Y^M)(W) = \nabla^M_V(\Y^N(W)) + \nabla^M_V(\Y^\perp(W))  - \Y^M(\nabla^N_VW) - \Y^M({\rm II}(V,W)),\]obtained via the definitions of $\Y^\perp$ and $\nabla^M\Y^M$, onto the tangent and normal bundles of $N$, using \eqref{Gauss-Weingarten}. From
  \[
    \begin{split}
      (\nabla^M_V\Y^M)(W)^\top &= \nabla^N_V(\Y^N(W)) - S_{\Y^\perp(W)}(V) - \Y^N(\nabla^N_VW) - \Y^M({\rm II}(V,W))^\top \\ (\nabla^M_V\Y^M)(W)^\perp &= {\rm II}(V,\Y^N(W)) + \nabla_V^\perp(\Y^\perp(W)) - \Y^\perp(\nabla^N_VW) - \Y^M({\rm II}(V,W))^\perp,
    \end{split}
  \] it suffices to recognize the definitions of $\nabla^N\Y^N$ and $\nabla^\perp\Y^\perp$, and evaluate all of it at the point $x$ to get (i) and (ii).
\end{proof}

With the definition of the normal covariant derivative $\nabla^\perp\Y^\perp$ in place, we also introduce the normal covariant derivative of the $s$-magnetic second fundamental form as
\begin{equation}\label{cov_mag-II}
  (\nabla^\perp_z {\rm II}^\sigma_s)_{(x,v)}(w) = s^2 (\nabla^\perp_z{\rm II})_x(v,w)  -s (\nabla^\perp_z\Y^\perp)_x(w),
\end{equation}
for all $(x,v) \in SN$, $w\in E^N_{(x,v)}$, and $z\in T_xN$. We may now proceed.

\begin{prop}\label{decomp-Rgs}
  The formulas
  \[
    \begin{split}
  {\rm (i)}~    (R^{\g,\sigma}_s)^M_{(x,v)}(w)^\top &= (R^{\g,\sigma}_s)^N_{(x,v)}(w)  -S_{({\rm II}^\sigma_s)_{(x,v)}(v)}(w)+ (S^\sigma_s)_{(x,v)}({\rm II}_x(v,w)) \\  &\quad - \frac{s}{2} {\rm P}_{v^\perp}(S_{\Y^\perp_x(w)}(v) + \Y^M_x({\rm II}_x(v,w))^\top) \\  {\rm (ii)}~ (R^{\g,\sigma}_s)^M_{(x,v)}(w)^\perp &= (\nabla^\perp_w{\rm II}^\sigma_s)_{(x,v)}(v) - (\nabla^\perp_v{\rm II}^\sigma_s)_{(x,v)}(w) - \frac{s}{2} (\nabla^\perp_v\Y^\perp)_x(w) \\ &\quad - s{\rm II}_x(w, \Y^N_x(v)) + \frac{s}{2} \Y^M_x({\rm II}_x(v,w))^\perp + \frac{s}{2} {\rm II}_x(v,\Y^N_x(w))
    \end{split}
\]hold, for any $(x,v) \in SN$, $w\in E^N_{(x,v)}$, and $s>0$.
\end{prop}

\begin{proof}
  Taking the tangent projection of both sides of \mbox{(\ref{eqn:AOmega}-ii)} for $(R^{\g,\sigma}_s)^M$ and applying the Gauss equation \eqref{Gauss-eqn}, formula (\ref{all-projectors}-ii), and Lemma \ref{cov-der-YM}${\rm (i)}$ with the roles of $v$ and $w$ switched, we obtain
  \[
    \begin{split}
      (R^{\g,\sigma}_s)_{(x,v)}^M(w)^\top &= s^2(R^N_x(w,v)v - S_{{\rm II}_x(v,v)}(w) + S_{{\rm II}_x(w,v)}(v)) \\  &\quad -s((\nabla^N_w\Y^N)_x(v) - S_{\Y^\perp_x(v)}(w) - \Y^M_x({\rm II}_x(v,w))^\top) \\ &\quad +\frac{s}{2} {\rm P}_{v^\perp}\left((\nabla^N_v\Y^N)_x(w) - S_{\Y^\perp_x(w)}(v) - \Y^M_x({\rm II}_x(v,w))^\top\right).
    \end{split}
  \]Simply recognizing the defintions of $(R^{\g,\sigma}_s)_{(x,v)}^N$, ${\rm II}^\sigma_s$, and $S^\sigma_s$, we obtain (i). As for (ii), we proceed in a similar fashion: taking the normal projection of both sides of \mbox{(\ref{eqn:AOmega}-ii)} for $(R^{\g,\sigma}_s)^M$, applying the Codazzi-Mainardi equation \eqref{CM-eqn}, and Lemma \ref{cov-der-YM}${\rm (ii)}$ with the roles of $v$ and $w$ switched, it follows that
\[
  \begin{split}
    (R^{\g,\sigma}_s)_{(x,v)}^M(w)^\perp &= s^2 ((\nabla^\perp_w{\rm II})_x(v,v) - (\nabla^\perp_v{\rm II})_x(w,v)) \\ &\quad -s((\nabla^\perp_w\Y^\perp)_x(v) +{\rm II}_x(w,\Y^N_x(v)) - \Y^M_x({\rm II}_x(w,v))^\perp) \\ &\quad+\frac{s}{2} {\rm P}_{v^\perp}\left( (\nabla^M_v\Y^M)_x(w)\right)^\perp.
  \end{split}
\]Expanding $(\nabla^M_v\Y^M)_x(w) = (\nabla^M_v\Y^M)_x(w)^\top + (\nabla^M_v\Y^M)_x(w)^\perp$, we may use formulas (\ref{all-projectors}-i) and (\ref{all-projectors}-ii) to conclude that ${\rm P}_{v^\perp}\left( (\nabla^M_v\Y^M)_x(w)\right)^\perp = (\nabla^M_v\Y^M)_x(w)^\perp$, so that another application of Lemma \ref{cov-der-YM}${\rm (ii)}$ yields
\[
  \begin{split}
    (R^{\g,\sigma}_s)_{(x,v)}^M(w)^\perp &= s^2 ((\nabla^\perp_w{\rm II})_x(v,v) - (\nabla^\perp_v{\rm II})_x(w,v)) \\ &\quad -s((\nabla^\perp_w\Y^\perp)_x(v) +{\rm II}_x(w,\Y^N_x(v)) - \Y^M_x({\rm II}_x(w,v))^\perp) \\ &\quad+\frac{s}{2} \left((\nabla^\perp_v\Y^\perp)_x(w) + {\rm II}_x(v,\Y^N_x(w)) - \Y^M_x({\rm II}_x(v,w))^\perp\right).
  \end{split}
\]Recognizing two instances of \eqref{cov_mag-II} and combining the terms with $\Y^M_x({\rm II}_x(v,w))^\perp$, (ii) is established.
\end{proof}

\begin{proof}[Proof of the main results]
Theorem \ref{mag-comp} follows from adding the expressions obtained in Propositions \ref{decomp_Ags} and \ref{decomp-Rgs}. Corollary \ref{mag-sec-Gauss} follows immediately from Theorem \ref{mag-comp} once one notes that the expression in the second line of item (i) in Proposition \ref{decomp-Rgs} is always orthogonal to $w$. As for item (i) of Corollary \ref{cor-hyp}, we substitute that ${\rm II}_x(v,w) = \g_x(S_\eta(v),w) \eta_x$ and $({\rm II}^\sigma_s)_{(x,v)}(w) = \g_x((S^\sigma_s)_{(x,v)}(\eta_x), w)\eta_x$ into Corollary \ref{mag-sec-Gauss}. Finally, in item (ii), we let $e_2,\ldots, e_{n-1}$ be an orthonormal basis of $E^N_{(x,v)}$, set $w=e_j$ in item (i) and sum over $j$, adding and subtracting the term $({\rm sec}^{\g,\sigma}_s)^M_x(v,\eta_x)$ required to produce $({\rm Ric}^{\g,\sigma}_s)^M(x,v)$.  
\end{proof}

\appendix

\section{The curvature of Killing magnetic systems}\label{app-killing}

We briefly explore some properties of Killing magnetic systems.

\subsection{Geometry in dimension three}

Let $(M,\g)$ be an oriented three-di\-men\-si\-o\-nal Riemannian manifold, and fix its Riemannian volume form $\mu_\g \in \Omega^3(M)$. Given $x\in M$ and $v,w\in T_xM$, we define the cross product $v\times w \in T_xM$ as the unique tangent vector such that
\begin{equation}\label{defn-cross}
  \parbox{.58\textwidth}{$\g_x(v\times w,u) = (\mu_\g)_x(v,w,u)$,\quad for all $u\in T_xM$.}
\end{equation}
This generalizes the standard cross product in $\R^3$, and shares all of its well-known algebraic properties. Namely, the $\times$ operation is bilinear and skew-symmetric, $v\times w$ is always orthogonal to both $v$ and $w$, we have that $v\times \cdot\in \mathfrak{so}(T_xM,\g_x)$ for each $v\in T_xM$, and the relations
\begin{equation}\label{double-cross}
\parbox{.9\textwidth}{\begin{enumerate}[(i)] \item $v\times (w\times z) = \g_x(v,z)w - \g_x(v,w)z$, \item $\g_x(v\times w, z) = \g_x(v,w\times z)$, and \item $\g_x(v_1\times v_2,w_1\times w_2) = \g_x(v_1,w_1)\g_x(v_2,w_2) - \g_x(v_1,w_2)\g_x(v_2,w_1)$\end{enumerate}}
\end{equation}
hold for all possible tangent vectors.

On the level of vector fields, if $X,Y,Z\in \mathfrak{X}(M)$, we have the Leibniz rule $\nabla_X(Y\times Z) = \nabla_XY\times Z + Y\times \nabla_XZ$ (as a consequence of $\g$ and $\mu_{\g}$ being \mbox{$\nabla$-parallel}), while $[X,Y\times Z] = [X,Y]\times Z + Y\times [X,Z]$ holds whenever $X$ is a Killing vector field (for a similar reason, as then $\mathcal{L}_X\g = \mathcal{L}_X\mu_\g = 0$). The curl of $X\in \mathfrak{X}(M)$ is the unique vector field ${\rm curl}_\g X \in \mathfrak{X}(M)$ such that
\begin{equation}\label{def-curl}
  \parbox{.83\textwidth}{$\mu_\g({\rm curl}_{\g}X, Y,Z) = \g(\nabla_YX,Z) -\g(Y,\nabla_ZX)$,\quad for all $Y,Z\in \mathfrak{X}(M)$.}
\end{equation}
Once again, this generalizes the classical notion of curl in $\R^3$. Gradients are always curl-free and curls are always di\-ver\-gen\-ce-free, as a consequence of the identity ${\rm d}^2=0$, and for any $f\in {\rm C}^\infty(M)$ it holds that ${\rm curl}_\g(fX) = f\,{\rm curl}_{\g}X + {\rm grad}_\g f \times X$. Finally, it follows from the Koszul formula together with \eqref{defn-cross} and \eqref{def-curl} that
\begin{equation}\label{Koszul-curl}
  \nabla_YX = \frac{1}{2}\,({\rm curl}_\g X) \times Y,\quad\mbox{for all }Y\in \mathfrak{X}(M),
\end{equation}whenever $X$ is a Killing vector field.

\subsection{Killing systems}

Given a vector field $\vec{B} \in \mathfrak{X}(M)$, we may consider the $2$-form $\sigma = \mu_\g(\vec{B},\cdot,\cdot)$ on $M$. As ${\rm d}\sigma = ({\rm div}_\g\vec{B})\,\mu_\g$, it follows that $\sigma$ is closed if and only if ${\rm div}_\g\vec{B} = 0$ (Gauss's law \emph{redux}), which in particular happens when $\vec{B} = {\rm curl}_{\g}X$ for some $X\in\mathfrak{X}(M)$, or when $\vec{B}$ is a Killing vector field; we assume from here onwards that we are in the latter situation (hence \emph{Killing systems}). It immediately follows from \eqref{defn:Lorentz-force} and \eqref{defn-cross} that the Lorentz force operator of $(\g,\sigma)$ is given by $\Y_x(v) = \vec{B}_x\times v$, for all $(x,v) \in TM$. Using (\ref{double-cross}-i) to deal with the last term in (\ref{eqn:AOmega}-i), we may readily compute that
\begin{equation}
  A^{\g,\sigma}_{(x,v)}(w) = -\frac{3}{4} \g_x(v,\vec{B}_x\times w) \,\vec{B}_x\times v - \frac{1}{4} \g_x(\vec{B}_x,w) {\rm P}_{v^\perp}(\vec{B}_x) + \frac{1}{4} \|\vec{B}_x\|^2w
\end{equation}for all $(x,v) \in SM$ and $w\in v^\perp$, where ${\rm P}_{v^\perp}$ is as in Section \ref{sec:magnetic_curvature}. We may also use \eqref{Koszul-curl} and (\ref{double-cross}-i) to compute that \begin{equation}  (\nabla_w\Y)_x(v) = \nabla_w\vec{B} \times v= \frac{1}{2} (({\rm curl}_{\g}\vec{B})_x\times w)\times v = \frac{1}{2}\g_x(({\rm curl}_\g\vec{B})_x,v)w, \end{equation}directly leading to
\begin{equation}
  (R^{\g,\sigma}_s)_{(x,v)}(w) = s^2 R_x(w,v)v - \frac{s}{2}\g_x(({\rm curl}_\g\vec{B})_x,v)w
\end{equation}from (\ref{eqn:AOmega}-ii), and hence to $(\vec{M}^{\g,\sigma}_s)_{(x,v)}(w)$ via \eqref{def:MOmega}. With this in place, it follows from (\ref{double-cross}-i) and \eqref{def:sec_and_Ric_mag} that the $s$-magnetic sectional curvature is given by
\begin{equation}\label{mag-sec-Killing}
 \begin{split} ({\rm sec}^{\g,\sigma}_s)_x(v,w) &= s^2{\rm sec}_x(v,w) - \frac{s}{2}\g_x(({\rm curl}_\g\vec{B})_x, v) \\ &\qquad+ \g_x(\vec{B}_x,v\times w)^2 + \frac{1}{4} \g_x(\vec{B}_x,v)^2 \end{split}
\end{equation}for all $(x,(v,w)) \in {\rm St}_2(M,\g)$. As for the $s$-magnetic Ricci tensor, note that $w$ and $v\times w$ form an orthonormal basis of $v^\perp$ whenever $(x,v) \in SM$, so that \eqref{def:sec_and_Ric_mag} yields \begin{equation}\label{pre-mag-Ric} \begin{split} ({\rm Ric}^{\g,\sigma}_s)(x,v) &= \g_x((\vec{M}^{\g,\sigma}_s)_{(x,v)}(w),w) + \g_x((\vec{M}^{\g,\sigma}_s)_{(x,v)}(v\times w),v\times w) \\ &= ({\rm sec}^{\g,\sigma}_s)_x(v,w) +({\rm sec}^{\g,\sigma}_s)_x(v,v\times w). \end{split}\end{equation} The second term is obtained by replacing $w$ and $v\times w$ in \eqref{mag-sec-Killing} with $v\times w$ and $v\times (v\times w) = -w$, respectively, cf. (\ref{double-cross}-i). It follows from \eqref{mag-sec-Killing} and \eqref{pre-mag-Ric} that
\begin{equation}\label{mag-Ric-Killing}
  ({\rm Ric}^{\g,\sigma}_s)(x,v) = s^2{\rm Ric}_x(v,v) - s\g_x(({\rm curl}_\g\vec{B})_x, v) + \|\vec{B}_x\|^2 - \frac{1}{2} \g_x(\vec{B}_x,v)^2
\end{equation}for all $(x,v)\in SM$.

\subsection{The static case} Here, we assume that the Killing field $\vec{B}$ is \emph{hy\-per\-sur\-fa\-ce-orthogonal}, in the sense that the distribution $\vec{B}^\perp$ on $M$ is nontrivial and integrable (in general relativity, hypersurface-orthogonal Killing fields are often called \emph{static}). Provided that $\vec{B}$ is nowhere-vanishing, \begin{equation}\label{B-curlB-ort} \parbox{.77\textwidth}{$\vec{B}^\perp$ is integrable if and only if ${\rm curl}_\g\vec{B}$ is always orthogonal to $\vec{B}$.}\end{equation}Indeed, this is an easy consequence of the formula \begin{equation}\label{int-curl} X \g(Y,\vec{B}) - Y\g(X,\vec{B}) = \g([X,Y],\vec{B}) + \g({\rm curl}_{\g}\vec{B}, X\times Y),\end{equation}valid for all $X,Y\in \mathfrak{X}(M)$, and obtained via \eqref{defn-cross}, \eqref{def-curl}, and \eqref{Koszul-curl}.

Let $N$ be an integral surface of $\vec{B}^\perp$, equipped with its induced magnetic system. We first claim that
\begin{equation}\label{int-surface-tot}
  \parbox{.85\textwidth}{$N$ is both totally geodesic and totally $s$-magnetic in $M$, for every $s>0$.}
\end{equation}
Namely, as $\Y^M_x(v) = \vec{B}_x\times v$ is always tangent to $N$ (being orthogonal to $\vec{B}_x$), it follows that $\Y^\perp = 0$; by Example \ref{tot-s-mag}, it now suffices to argue that $N$ is totally geodesic. This, in turn, is immediate from $\vec{B}$ being a Killing vector field and $N$ being an integral surface of $\vec{B}^\perp$: geodesics which start tangent to $N$ remain tangent to $N$ as the function $t\mapsto \g_{\gamma(t)}(\dot{\gamma}(t),\vec{B}_{\gamma(t)})$ is constant whenever $\gamma$ is a geodesic but vanishes at $t=0$.

It now follows from \eqref{int-surface-tot} that both the shape operators and the $s$-magnetic shape operators of $N$ vanish. Hence, with the aid of (\ref{double-cross}-iii), \eqref{Koszul-curl}, and \eqref{B-curlB-ort}, we use Corollary \ref{cor-hyp} to compute the $s$-magnetic Gaussian curvature $K^{\g,\sigma}_s\colon SN \to \R$ of $N$ as
\begin{equation}
  \begin{split}
      K^{\g,\sigma}_s(x,v) &= ({\rm sec}^{\g,\sigma}_s)^N_x(v, \|\vec{B}_x\|^{-1}\vec{B}_x\times v) \\ &\hspace{-1pt}= s^2{\rm sec}^M(T_xN) - \frac{s}{2} \g_x(({\rm curl}_\g\vec{B})_x, v)   + \|\vec{B}_x\|^2.
    \end{split}    
  \end{equation}Particularizing it further: if $\vec{B}$ has constant length $b = \|\vec{B}\|\neq 0$, then $\vec{B}$ is a geodesic vector field, in the sense that $\nabla_{\vec{B}}\vec{B} = 0$, and this implies via \eqref{Koszul-curl} that ${\rm curl}_{\g}\vec{B}=0$ (as ${\rm curl}_\g\vec{B}$ becomes both orthogonal and proportional to $\vec{B}\neq 0$). Hence, $\vec{B}$ is parallel (again by \eqref{Koszul-curl}) and $K^{\g,\sigma}_s(x,v) = s^2 {\rm sec}^M(T_xN) + b^2$ is independent of $v$.

  \subsection{Special Killing systems on $\esf^3$}

  Denoting by $x=(x^0,x^1,x^2,x^3)$ the usual Cartesian coordinates in $\R^4$, consider the unit sphere \[\esf^3 = \{ x\in \R^4 \colon (x^0)^2+(x^1)^2+(x^2)^2+(x^3)^2=1\}\]equipped with its standard round metric $\g^\circ$. Its tangent bundle is globally trivialized by the orthonormal vector fields
  \begin{equation}\label{Ei-Ej-Ek}
    \begin{split}
      \vec{E}_{\rm i}(x) &=  -x^1 \partial_0+x^0 \partial_1 - x^3\partial_2 + x^2\partial_3 \\ \vec{E}_{\rm j}(x) &=  -x^2\partial_0 + x^3 \partial_1 + x^0 \partial_2 - x^1 \partial_3 \\ \vec{E}_{\rm k}(x) &= -x^3\partial_0 - x^2\partial_1 + x^1\partial_2 + x^0\partial_3,
    \end{split}
  \end{equation}corresponding to quaternionic left-multiplication by ${\rm i}$, ${\rm j}$, and ${\rm k}$; their Lie brackets are easily computed to be
  \begin{equation}\label{Lie-B-ijk}
    [\vec{E}_{\rm i},\vec{E}_{\rm j}] = -2\vec{E}_{\rm k},\quad     [\vec{E}_{\rm j},\vec{E}_{\rm k}] = -2\vec{E}_{\rm i},\quad\mbox{and}\quad     [\vec{E}_{\rm k},\vec{E}_{\rm i}] = -2\vec{E}_{\rm j}.
  \end{equation}
  In addition, \eqref{Ei-Ej-Ek} are Killing fields for $\g^\circ$, as their flows consist of rotations in $\esf^3$. Orienting $\esf^3$ so its Riemannian volume form is \mbox{$\mu_{\g^\circ} = \theta^{\rm i}\wedge \theta^{\rm j}\wedge \theta^{\rm k}$}, where $\theta^{\rm i},\theta^{\rm j},\theta^{\rm k}$ constitute the coframe dual to \eqref{Ei-Ej-Ek}, we may use \eqref{Lie-B-ijk} and \eqref{int-curl} to obtain that
  \begin{equation}\label{curl-Ei}
    {\rm curl}_\g\vec{E}_{\rm i} = 2\vec{E}_{\rm i},\quad     {\rm curl}_\g\vec{E}_{\rm j} = 2\vec{E}_{\rm j},\quad\mbox{and}\quad     {\rm curl}_\g\vec{E}_{\rm k} = 2\vec{E}_{\rm k}.
  \end{equation}

  Consider now the Killing magnetic system $(\g^\circ,\sigma^{\rm i})$ corresponding to \mbox{$\vec{B} = \vec{E}_{\rm i}$}. Observe that $\vec{E}_i$ is \emph{not} hypersurface-orthogonal, cf. \eqref{B-curlB-ort} and \eqref{curl-Ei}, and that $\sigma^{\rm i} = {\rm d}(\theta^{\rm i}/2) = \theta^{\rm j} \wedge \theta^{\rm k}$ as a consequence of \eqref{Lie-B-ijk}. It is well-known that the corresponding magnetic geodesics are the helices whose axes are the great circles of $\esf^3$ tangent to $\vec{E}_{\rm i}$ \cite[Theorem 4.1]{Cabrerizo_2013}. Combining \eqref{curl-Ei} with \eqref{Koszul-curl} to write \mbox{$\nabla_w\vec{E}_{\rm i} = \vec{E}_{\rm i}(x)\times w$}, we compute the $s$-magnetic sectional curvature via \eqref{mag-sec-Killing} to be
  \begin{equation}\label{mag-sec-S3}
    \big({\rm sec}^{\g^\circ,\sigma^{\rm i}}_s\big)_x(v,w) = \left(s-\frac{1}{2} \g_x(\vec{E}_{\rm i}(x),v)\right)^2 + \g_x(\vec{E}_{\rm i}(x), v\times w)^2,
  \end{equation}
  for all $(x,(v,w))\in {\rm St}_2(\esf^3,\g^\circ)$. It is clear that \eqref{mag-sec-S3} is non-negative for all values of $s>0$, and it is in fact positive whenever $s> 1/2$. On the other hand, if $s\in (0,1/2]$, we may choose any $(x,v) \in S\esf^3$ such that $\g_x(\vec{E}_{\rm i}(x),v)=2s$, and let $w$ be any unit vector orthogonal to both $v$ and $\vec{E}_{\rm i}(x)\times v$, so that $({\rm sec}^{\g^\circ,\sigma^{\rm i}}_s)_x(v,w) =0$. (For example, $v = 2s\vec{E}_{\rm i}(x) + (1-4s^2)^{1/2}\vec{E}_{\rm j}(x)$ and $w = (1-4s^2)^{1/2}\vec{E}_{\rm i}(x)-2s\vec{E}_{\rm j}(x)$ work.)

  In a similar manner, using \eqref{curl-Ei} and \eqref{mag-Ric-Killing}, it follows that each \mbox{$s$-magnetic} Ricci curvature is given by
  \begin{equation}\label{mag-Ric-S3}
    ({\rm Ric}^{\g^\circ,\sigma^{\rm i}}_s)(x,v) = 2s^2 - 2s\g_x(\vec{E}_{\rm i}(x),v) + 1 - \frac{1}{2}\g_x(\vec{E}_{\rm i}(x),v)^2,
  \end{equation}
  for all $(x,v) \in S\esf^3$. This time, we have that \eqref{mag-Ric-S3} is positive for all values $s \in (0,\infty)\smallsetminus \{1/2\}$, while $({\rm Ric}^{\g^\circ, \sigma}_{1/2})(x,\vec{E}_{\rm i}(x)) = 0$. See Figure \ref{fig:lower-mag-Ric} below. This was predicted by item (ii) of \cite[Proposition 7]{Assenza_2024}, which states that if $(\g,\sigma)$ is any magnetic system on a compact manifold $M$, and $\sigma$ is nowhere-vanishing, there is $s_0>0$ such that ${\rm Ric}^{\g,\sigma}_s>0$ for all $s\in (0,s_0)$.

  \begin{figure}[H]
    \centering
    \includegraphics[scale=.8]{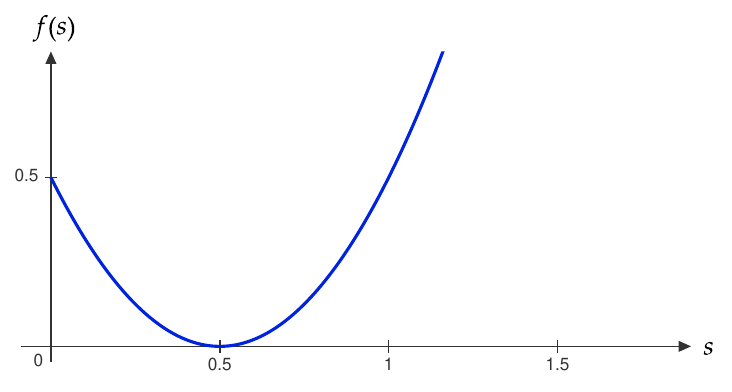}
    \caption{The graph of the function $f\colon (0,\infty)\to \R$ given by $f(s) = ({\rm Ric}^{\g^\circ,\sigma^{\rm i}}_s)(x,\vec{E}_{\rm i}(x)) = 2s^2-2s+1/2 = 2(s-1/2)^2$, satisfying the optimal bound ${\rm Ric}^{\g^\circ,\sigma^{\rm i}}_s \geq f(s)$ for every $s>0$.}
    \label{fig:lower-mag-Ric}
  \end{figure}

  The value $s_0=1/2$ marks a distinct change in the geometric behavior of the magnetic system $(\g^\circ, \sigma^{\rm i})$. This makes our last result not surprising:

  \begin{prop}\label{Mane-S3}
    The Ma\~{n}\'{e} critical value is $\mathbf{c}(\g^\circ,\sigma^{\rm i})=1/2$.
  \end{prop}

  Recall from Section \ref{sec:magnetic_curvature} that the dynamics of the magnetic flow on the $s$-sphere bundle depends on the value of $s>0$. The Ma\~{n}\'{e} critical value $\mathbf{c}(\g,\sigma)$---defined for arbitrary magnetic systems $(\g,\sigma)$ on any manifold $M$---marks where the dynamical behavior of the magnetic flow changes. To justify Proposition \ref{Mane-S3}, we use two distinct characterizations of $\mathbf{c}(\g,\sigma)$. The first one is that
  \begin{equation}\label{Mane-1}
    \mathbf{c}(\g,\sigma) =
    \begin{cases}
      \inf\limits_{p^*\sigma = {\rm d}\theta}\sup\limits_{x\in \widetilde{M}} \|\theta_x\|,&\mbox{ if $\sigma$ is weakly exact}, \\ +\infty,&\mbox{otherwise},
    \end{cases}
  \end{equation}where $p\colon \widetilde{M}\to M$ is the universal covering of $M$, cf. \cite[Section 2]{Merry_2010}. The second one, in the weakly exact case, is that
  \begin{equation}\label{Mane-2}
    \mathbf{c}(\g,\sigma) = \inf\{s>0 : A_s(\gamma) \geq 0\mbox{ for all }\gamma \in \Lambda\widetilde{M} \},
  \end{equation}where $\Lambda\widetilde{M}$ is the Hilbert manifold of absolutely continuous loops in $\widetilde{M}$, and the action functional $A_s\colon \Lambda \widetilde{M} \to \R$ is defined as \begin{equation}\label{action-As}  A_s(\gamma) = \int_0^T \frac{1}{2}(\|\dot{\gamma}(t)\|^2+s^2) - \theta_{\gamma(t)}(\dot{\gamma}(t))\,{\rm d}t,  \end{equation}with $T>0$ being the least period of $\gamma$, the norm computed via $p^*\g$, and $\theta$ being any primitive of $p^*\sigma$; see \cite[Section 5.1]{CieliebakFrauenfelderPaternain_2010}.

  \begin{proof}[Proof of Proposition \ref{Mane-S3}]
 As $\sigma^{\rm i} = {\rm d}(\theta^{\rm i}/2)$ and $\|\theta^{\rm i}/2\|=1/2$, it follows from \eqref{Mane-1} that $\mathbf{c}(\g^\circ ,\sigma^{\rm i}) \leq 1/2$. As for the reverse inequality, let $s\in (0,1/2)$ and consider the curve $\gamma\colon [0, 2\pi/s]\to \esf^3$ given by $\gamma(t) = (\cos(st), \sin(st),0,0)$. Then $\theta^{\rm i}_{\gamma(t)}(\dot{\gamma}(t)) = s$ implies---via \eqref{action-As}---that $A_s(\gamma) = (2s-1)\pi < 0$, due to our choice of $s$; by \eqref{Mane-2}, this shows that $\mathbf{c}(\g^\circ,\sigma^{\rm i}) \geq 1/2$.
  \end{proof}

\bibliography{magnetic_refs}{}
\bibliographystyle{plain}

\end{document}